\documentclass{article}
\usepackage[margin=1.2in]{geometry}
\usepackage{graphicx}
\usepackage{amssymb}
\usepackage{amsmath}
\usepackage{amsfonts}
\usepackage{amsthm}
\usepackage{amsmath,amscd}
\usepackage{hyperref}
\begin{document}

\title{\textbf{The genus zero, 3-component fibered links in $S^3$}}
\author{\textbf{Carson Rogers}}
\date{ }
\maketitle

\graphicspath{{./gen0_3comp_images/}}
\theoremstyle{plain}
\newtheorem{thm}{\textbf{Theorem}}[section]
\newtheorem{lem}{\textbf{Lemma}}[section]
\theoremstyle{definition}
\newtheorem{defn}{\textbf{Definition}}[section]
\newtheorem{propn}{\textbf{Proposition}}[section]
\newtheorem*{prodisk}{\textbf{Product disks}}
\newtheorem*{fib03}{\textbf{Fiber surfaces of type (0,3)}}
\newtheorem*{stallings}{\textbf{Stallings twists}}
\newtheorem*{heegfib}{\textbf{Heegaard diagrams from fibered links}}
\newtheorem*{whitewave}{\textbf{Whitehead graphs and waves}}

\begin{abstract}
The open book decompositions of the 3-sphere whose pages are pairs of pants have been fully understood for some time, through the lens of contact geometry. The purpose of this note is to exhibit a purely topological derivation of the classification of such open books, in terms of the links that form their bindings and the corresponding monodromies. We construct all of the links and their pair-of-pants fiber surfaces from the simplest example, a connected sum of two Hopf links, through performing (generalized) Stallings twists. Then, by applying the  now-classical theory of genus two Heegaard diagrams in $S^3$, we verify that the monodromies of the links in this family are the only ones corresponding to pair-of-pants open book decompositions of $S^3$.
\end{abstract}

\section{Introduction}

Say that an open book decomposition of a closed 3-manifold $M$ is of \textit{type $(g,b)$} if the pages have genus $g$ and the binding is a $b$-component link. It comes as no surprise that $S^3$ admits very few open book decompositions of types $(g,b)$ for some of the smallest values of $g$ and $b$. Namely, the annular open book decompositions with bindings given by the right and left-handed Hopf links are the only open book decompositions of $S^3$ of type $(0,2)$ \cite{BG16}, and the two trefoil knots and the figure-eight knot form the bindings of the only three open book decompositions of type $(1,1)$ \cite{Go70}. \\

The next simplest open book decompositions of $S^3$ are those of types $(0,3)$, whose pages are pairs of pants. This is an infinite family of open books, and in the proof of Lemma 5.5 of \cite{EO08}, Etnyre and Ozbagci quickly enumerate the monodromies of all examples. Their goal is to describe the corresponding contact structures, and as such, their argument relies on contact geometry. In particular, they rule out a large class of potential monodromies by considering Stein fillings of the corresponding contact 3-manifolds.\\

Before becoming aware of their work, and with different motivations, the author arrived at the same explicit classification of these open books several years ago, through purely 3-dimensional, topological methods. While experts on fibered links and Heegaard theory likely expect that this is possible to do, it seems that no such argument has appeared in writing outside of the author's PhD thesis. The purpose of releasing this separately is, ideally, to make the result better known to low-dimensional topologists who are not immersed in contact geometry.\\

\begin{figure}[t]
	\centering
	\includegraphics[scale=0.4]{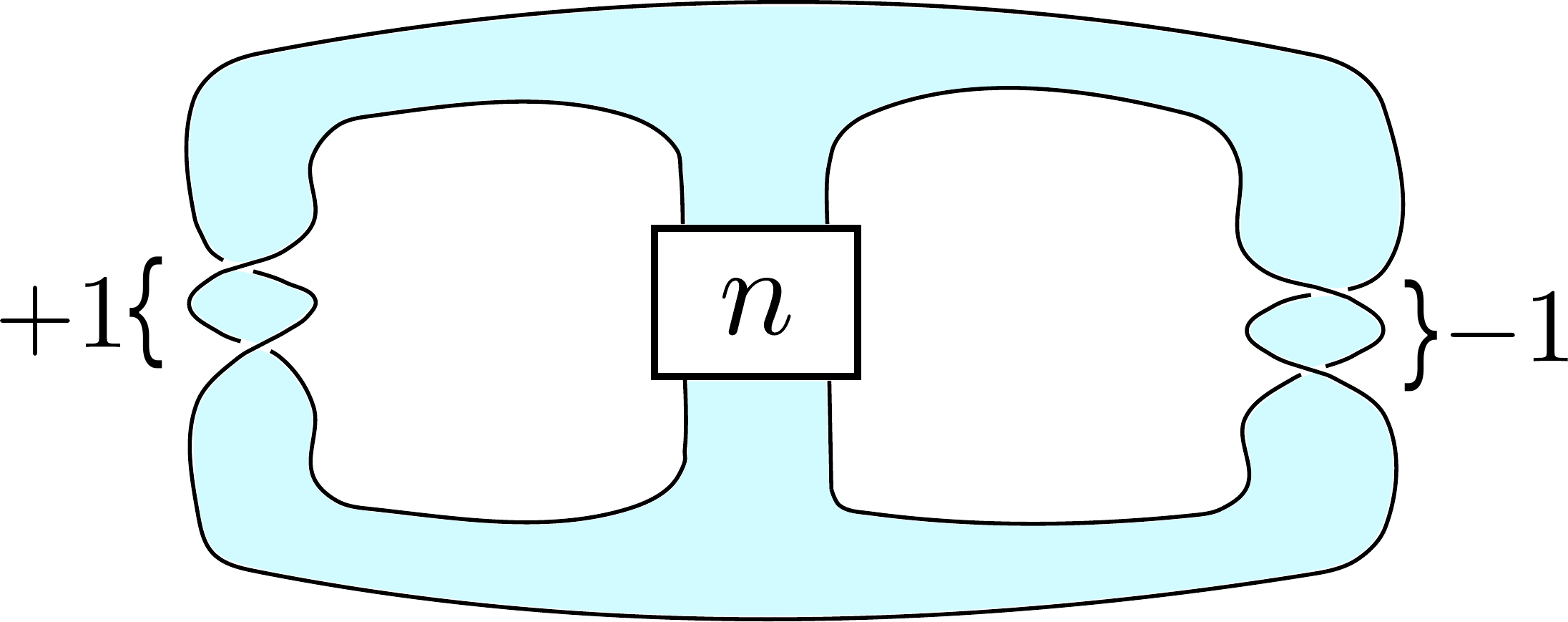}
	\caption{The links $L_n$, together with their genus zero fiber surfaces.}
	\label{gabai_surf}
\end{figure}

We couch the theorem and our proof in the language of fibered links. Say that a ($b$-component) fibered link in $S^3$ is of \textit{type $(g,b)$} if it forms the binding of an open book decomposition of that type. 

\begin{thm}\label{mainthm}
The fibered links in $S^3$ of type $(0,3)$ consist of connected sums of Hopf links, the links $L_n$ of Figure \ref{gabai_surf}, the link depicted in Figure \ref{gen_zero_link}, and its mirror image. The fiber surfaces of any two of these links are related by a sequence of Stallings twists.
\end{thm}

The second statement of this theorem is a byproduct of the most natural way of constructing the fiber surfaces of these links, so as to identify their monodromies. Here, we take the definition of a Stallings twist given by Harer \cite{Ha82}, in which the twisting curve on the fiber surface is allowed to have framing 0 or $\pm 2$. Following explicit descriptions of the monodromies of the links listed here, we prove that no other automorphism of a pair of pants can arise as the monodromy of such a fibered link in $S^3$ by analyzing the corresponding genus two Heegaard diagrams.\\

We note that the only redundancy to occur in our listing of these links is $L_0$, which is a connected sum of two Hopf links of opposite type. It is no accident that all of these links contain a Hopf sub-link: if $L$ is a 3-component link in a 3-manifold $M$ whose exterior is fibered by pairs of pants, then $L$ contains the exceptional fibers of a Seifert fibering of $M$ over $S^2$. \\

The paper is organized as follows. In Section 2, we give background and definitions. In Section 3.1, we provide constructions of the links described above which prove that they are fibered and of genus zero, when given appropriate orientations. In Section 3.2, we then use the theory of genus two Heegaard diagrams in $S^3$ to show that any fibered link in $S^3$ of type $(0,3)$ must have the same monodromy as one of these links.

\begin{figure}[b]
	\centering
	\includegraphics[scale=0.55]{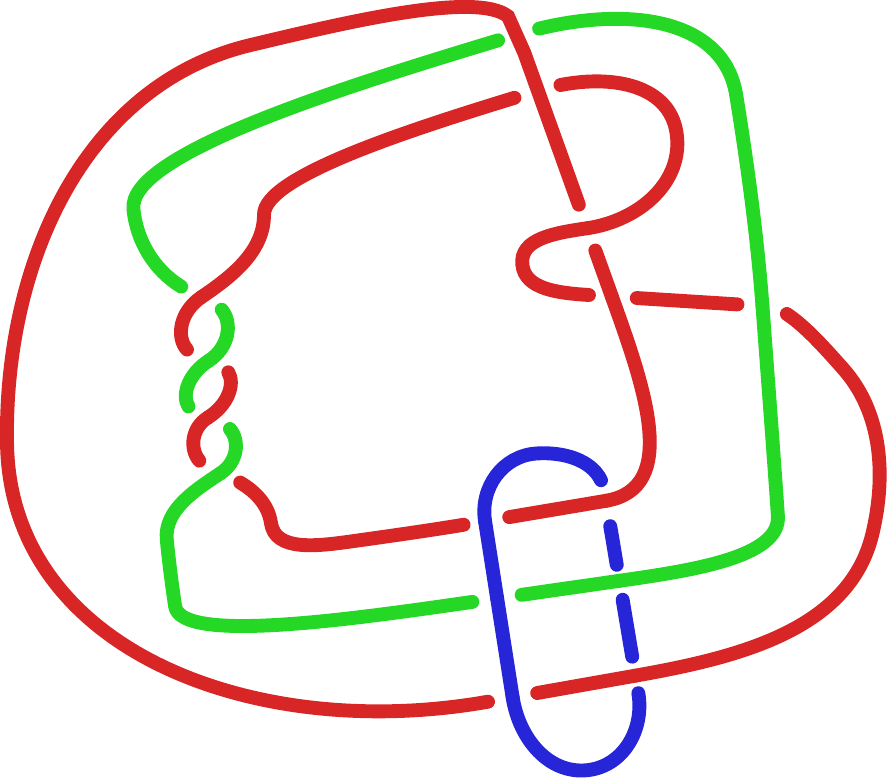}
	\caption{One of the two exceptional fibered links of type $(0,3)$, the other being its mirror image.}
	\label{gen_zero_link}
\end{figure}

\section{Preliminaries}

All manifolds under consideration are assumed to be orientable. Isotopies of embedded objects in a 3-manifold $M$ are always taken to be proper when $\partial M\neq\emptyset$. Two links or embedded surfaces in $M$ will be regarded as the same if they are (properly) isotopic. If $L$ is a link in $M$, we use $M_L$ to denote the compact exterior $\overline{M-\eta(L)}$ of $K$ in $M$, where $\eta(L)$ is a tubular neighborhood of $L$ in $M$. For basic terminology and facts from 3-manifold topology that are taken for granted throughout, we refer the reader to \cite{Sc14}.

\subsection{Fibered links}

Let $L$ be an oriented null-homologous link in a closed, oriented 3-manifold $M$. A \textit{Seifert surface} for $L$ is an embedded compact, oriented surface $F$ in $M$ such that $\partial F=L$. We can equally well regard a Seifert surface as being properly embedded in $M_L$, and the two viewpoints will be used interchangeably. We say that $L$ is \textit{fibered} if $E(L)$ fibers over $S^1$, so that the fiber $F$ is a Seifert surface for $L$. We say that $F$ is a \textit{fiber surface} for $L$. When speaking of a fiber surface without reference to a specific link, we mean a fiber surface for some fibered link in $M$. A compact, orientable surface is said to be of \textit{type $(g,b)$} if it has genus $g$ and $b$ boundary components. A fibered link in $M$ is said to be of type $(g,b)$ if its fiber surface is of type $(g,b)$.\\

If $L$ is a fibered link with fiber surface $F$, then $M_L$ may be realized as a mapping torus $F\times[0,1]/\left((x,1)\sim(\phi(x),0)\right)$, where $\phi:F\rightarrow F$ is a homeomorphism which fixes $\partial F$ pointwise. We say that $\phi$ is a \textit{monodromy} for $F$. Note that a monodromy for $F$ is defined up to isotopy and conjugation by other self-homeomorphisms of $F$.\\

Conversely, given a surface with boundary $F$ and an orientation-preserving homeomorphism $\phi:F\rightarrow F$ which fixes $\partial F$ pointwise, there is a canonical way \cite{Et06} to fill in the boundary of the mapping torus of $\phi$ with solid tori $V_1,\ldots,V_n$ to obtain a closed 3-manifold $M_{\phi}$. The cores of $V_1,\ldots,V_n$ then constitute a well-defined link $L_{\phi}$ in $M_{\phi}$, which is denoted by $B_{\phi}$ in \cite{Et06}.

\begin{prodisk}
	Let $L$ be a fibered link in $M$ with fiber surface $F$, and let $\eta(F)\cong F\times I$ be a closed regular neighborhood of $F$ in $M_L$ such that $F\times\{t\}$ is a fiber of the fibration of $M_L$ for each $t\in I$. The surface exterior $M_F=\overline{M_L-N(F)}$ has a corresponding parametrization as $F\times I$. A \textit{product disk} is in $\eta(F)$ or $M_F$ is a properly embedded disk $D$ such that $\partial D$ intersects each component of $F\times\partial I$ in a single properly embedded arc. A product disk may be isotoped so that it intersects each fiber $F\times\{t\}$ in a single arc.\\
	
	A monodromy $\phi$ for $F$ defines a natural pairing between product disks in $\eta(F)$ and those in $M_F$. Viewing $M_L$ as the quotient of $F\times[0,1]$ by the action of $\phi$, we take $F\times[0,1/2]$ to be $\eta(F)$ and $F\times[1/2,1]$ to be $M_F$. Up to isotopy, every product disk in $\eta(F)$ is uniquely realized as $D(\alpha)=\alpha\times[0,1/2]$ for some properly embedded arc $\alpha$ in $F$. The same is true in $M_F$, with $[0,1/2]$ replaced by $[1/2,1]$. Since $(x,1)$ is identified with $(\phi(x),0)$ to form $M_L$, we denote $\alpha\times[1/2,1]$ by $D(\phi(\alpha))$. Once perturbed to lie in general position, $\partial D(\alpha)$ and $\partial D(\phi(\alpha))$ will typically intersect non-trivially. 
\end{prodisk}

\begin{fib03}
	The types of fiber surfaces to be considered here, with their monodromies, are naturally enumerated by multisets of three integers. Let $F$ be an oriented pair of pants, an abstract surface of type $(0,3)$, with boundary components labeled $b_1,b_2,b_3$. For each $i=1,2,3$, let $c_i$ be a simple closed curve embedded in $Int(F)$ which is parallel to $b_i$. Assume that $c_1,c_2,c_3$ have been chosen to be mutually disjoint. Throughout, we let $T_i$ denote a left Dehn twist along $c_i$, according to the convention of Figure \ref{dehn}. The inverse $T^{-1}_i$ is then a right twist along $c_i$. The composition of $T^{n_i}_i$ and $T^{n_j}_j$ will be denoted by $T^{n_i}_iT^{n_j}_j$. Note that, since the three curves are mutually disjoint, $T_i$ always commutes with $T_j$.\\
	
	\begin{figure}[t]
		\centering
		\includegraphics[scale=0.5]{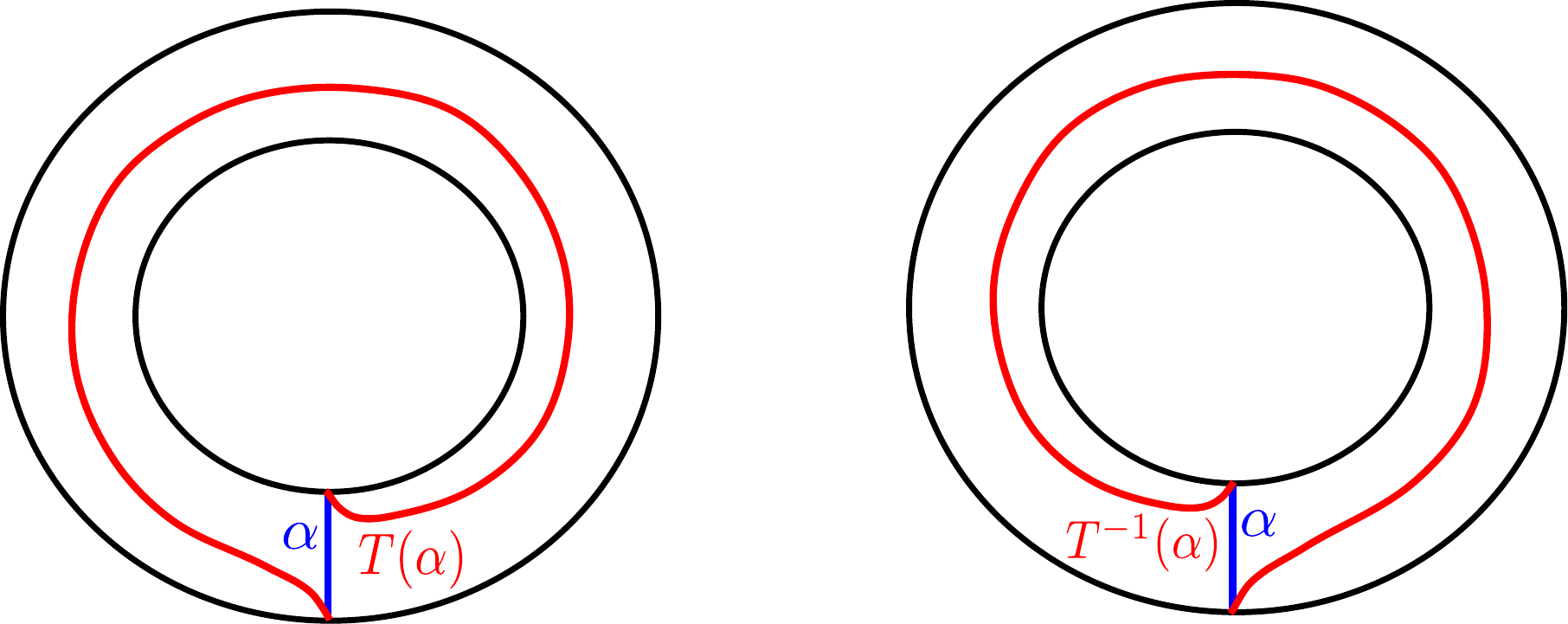}
		\caption{Standard conventions for right and left-handed Dehn twists (right and left, resp.)}
		\label{dehn}
	\end{figure}
	
	The mapping class group of $F$ is isomorphic to $\mathbb{Z}^3$ \cite{FM12}, with an explicit isomorphism given by sending the isotopy class of $T_i$ to the $i^{th}$ standard unit vector $e_i=(\delta_{1i},\delta_{2i},\delta_{3i})$. Every orientation-preserving homeomorphism $\phi:F\rightarrow F$ which fixes $\partial F$ pointwise is therefore isotopic to a unique homeomorphism of the form $T^{n_1}_1\circ T^{n_2}_2\circ T^{n_3}_3$. The manifold-link pair $(M_{\phi},L_{\phi})$ is then determined by the triple $n_1,n_2,n_3$. \\
	
	Further, if $\{n_1,n_2,n_3\}=\{n'_1,n'_2,n'_3\}$, then the mapping tori of $\phi=T^{n_1}_1T^{n_2}_2T^{n_3}_3$ and $\phi'=T^{n'_1}_1T^{n'_2}_2T^{n'_3}_3$ are exactly the same, up to relabeling the boundary components of $F$. The pairs $(M_{\phi},L_{\phi})$ and $(M_{\phi'},L_{\phi'})$ are therefore the same. From this point forward, we denote this pair by $(M[n_1,n_2,n_3],L[n_1,n_2,n_3])$, understanding that the order of $n_1$, $n_2$, and $n_3$ is unimportant. The associated fiber surface for $L[n_1,n_2,n_3]$ will likewise be denoted by $F[n_1,n_2,n_3]$.
\end{fib03}

\begin{stallings}
	Stallings twists are operations by which new fiber surfaces in $S^3$ can be produced from old ones. They were introduced by Stallings in \cite{St78}, and will be used to construct the fibered links appearing in the statement of the main theorem. The version described here is a slight generalization of Stallings' original definition which appears in \cite{Ha82}.\\
	
	Let $F$ be an oriented fiber surface in $S^3$ bounded by a link $L$, and let $c$ be an essential closed curve embedded in $F$ which is unknotted in $S^3$. Suppose that $lk(c,c^+)+\delta_1=\delta_2$, where $c^+$ is a copy of $c$ pushed off of $F$ to the positive side, $\delta_1=\pm 1$, and $\delta_2=\pm 1$. Let $A$ be a regular neighborhood of $c$ in $F$ and $N$ be a solid torus obtained by thickening $A$ to the positive side of $F$. We choose $N$ so that a neighborhood of $c^+$ in the corresponding fiber lies in $\partial N$, and the restriction of the fibration of $S^3-L$ to $N$ is a product fibration by annuli.\\
	
	\begin{figure}[b]
		\centering
		\includegraphics[scale=0.5]{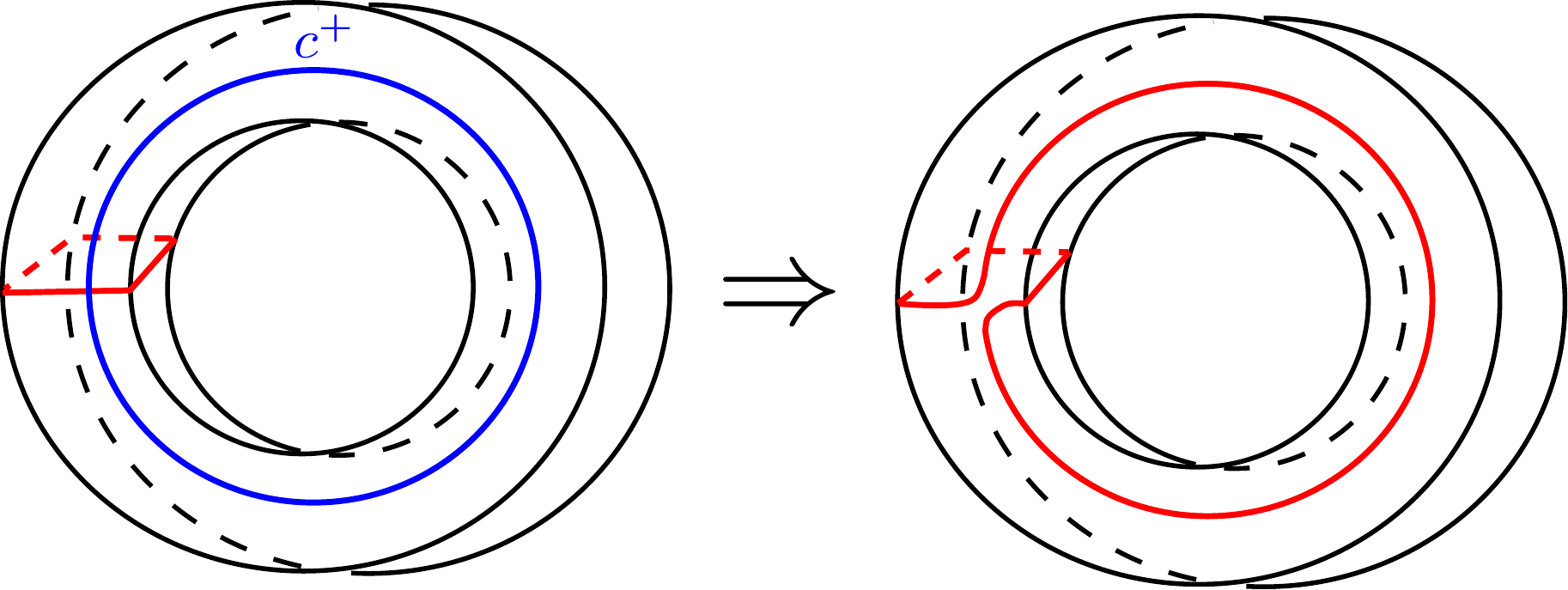}
		\caption{Re-gluing map for the solid torus $N$ which defines a Stallings twist.}
		\label{stal}
	\end{figure}
	
	To obtain a new fiber surface from $F$, perform Dehn surgery on $S^3$ along the core of $N$ by removing $N$ and regluing it by the homeomorphism of the boundary torus illustrated in Figure \ref{stal}. This map sends a meridian of $N$ which intersects $c^+$ once transversely to a Haken sum of itself with $c^+$. The choice of which way to resolve the point of intersection depends on $\delta_1$: we resolve it so that one `turns left' into the meridian after traversing $c^+$ if $\delta_1=-1$, and 'turns right' into the meridian if $\delta_1=+1$. \\
	
	It follows that the surgery coefficient is $\delta_2=\pm 1$, so the resulting 3-manifold is $S^3$. The key point is that this procedure turns $F$ into a new fiber surface $F'$, bounded by some link $L'$. Away from $N$, the fibration of the exterior of $L'$ will agree with that of the exterior of $L$. The restriction of the fibration to $N$ will again be a product fibration of $N$ by annuli, two being neighborhoods of $c$ and $c_+$ in $\partial N$. If $h$ is a monodromy for $F$, then $h\circ T^{\delta_1}_c$ is a monodromy for $F'$, where $T_c$ denotes a left-handed Dehn twist along $c$.
\end{stallings}

\begin{defn}
	In the above situation, we say that $F'$ is obtained from $F$ by performing a \textit{Stallings twist of type $\epsilon$} on $F$ along $c$, where $\epsilon=\frac{1}{2}lk(c,c^+)$. We also say that $c$ is a \textit{Stallings curve of type $\epsilon$} with respect to $F$. A Stallings twist will be called \textit{positive} (resp. \textit{negative}) if the surgery coefficient $\delta_2$ is $+1$ (resp. $-1$). 
\end{defn}

Note that the condition $lk(c,c^+)+\delta_1=\delta_2$ means that $\epsilon$ is either 0 or $\pm 1$. While a Stallings twist of type 0 can be either positive or negative, those of type 1 are necessarily negative, while those of type -1 must be positive. If $F'$ is obtained from $F$ by performing a Stallings twist of type $\epsilon$ along $c$, then $F$ is evidently obtained from $F'$ by performing a Stallings twist of type $-\epsilon$ along $c$, where $c$ is now viewed as a curve embedded in $F'$ in the natural sense. \\

A final observation is that a Stallings twist of type 0 may be iterated. If $c$ is a Stallings curve of type 0 in $F$ and $F'$ is obtained from $F$ by performing a Stallings twist along $c$, then $c$ is again a Stallings curve of type 0 in $F'$. It therefore makes sense to speak of performing a positive or negative Stallings twist along $c$ $n$ times in succession for any $n>0$.

\begin{defn}
	Suppose that $F$ is a fiber surface in $S^3$ and $c$ is a Stallings curve of type 0 in $F$. If $F'$ is the fiber surface obtained from $F$ by performing $n>0$ successive positive (resp. negative) Stallings twists along $c$, then we say that $F'$ is obtained from $F$ by performing a \textit{Stallings $n$-twist} (resp. \textit{$-n$-twist}) along $c$. 
\end{defn}

By induction and our previous observation on the effect of a Stallings twist on the monodromy, one sees that if $F'$ is obtained from $F$ by performing a Stallings $n$-twist along $c$ and $h$ is a monodromy for $F$, then $h\circ T^n_c$ is a monodromy for $F'$.

\subsection{Genus two Heegaard diagrams}

Throughout this subsection, $\Sigma$ will denote a closed, orientable surface of genus two. A \textit{curve} in $\Sigma$ will always refer to a simple closed curve embedded in $\Sigma$. Every curve under consideration is assumed to be \textit{essential}, meaning that it does not bound an embedded 2-disk in $\Sigma$. Whenever discussing a collection of two or more curves in $\Sigma$, we always assume that they have been isotoped so that every pair intersects transversely in a finite number of points. Upon orienting two curves $c_1$ and $c_2$ in $\Sigma$, we may consider the algebraic intersection number of $c_1$ with $c_2$, denoted $c_1\cdot c_2$, by counting the signed number of intersection points of $c_1$ with $c_2$.\\

A \textit{cut system} for $\Sigma$ is a pair $\alpha$ of disjoint curves $\alpha_1$ and $\alpha_2$ which cut $\Sigma$ into a planar surface, denoted by $\Sigma_{\alpha}$, with four boundary components. A (genus two) \textit{Heegaard diagram} $(\Sigma,\alpha,\beta)$ consists of two cut systems $\alpha$ and $\beta$ for $\Sigma$. We typically suppress $\Sigma$ and denote the Heegaard diagram by $(\alpha,\beta)$. Given a Heegaard diagram $(\alpha,\beta)$, we may construct a closed 3-manifold $N$ from $\Sigma\times[-1,1]$ by attaching one pair of 3-dimensional 2-handles along $\alpha\times\{-1\}$ and another along $\beta\times\{1\}$, and then capping off the resulting 2-sphere boundary components with 3-balls. We then say that $(\alpha,\beta)$ is a \textit{Heegaard diagram for $N$}. \\

Recall that a \textit{Heegaard surface} for a 3-manifold $M$ is a closed, orientable, separating surface $\Sigma$ embedded in $M$ such that the closure of each component of $M-\Sigma$ is a handlebody. In the above construction, the surface $\Sigma=\Sigma\times\{0\}$ is a genus two Heegaard surface for the 3-manifold $N$. Conversely, every closed 3-manifold admitting a genus two Heegaard surface arises from a Heegaard diagram via this construction.

\begin{heegfib}
	Let $L$ be a fibered link in a 3-manifold $M$ with fiber surface $F$. The boundary of a regular neighborhood of $F$ in $M$ divides $M$ into two copies of the handlebody $F\times I$, and is therefore a Heegaard surface for $M$. Given a monodromy $\phi$ for $F$, a corresponding Heegaard diagram can be constructed from product disks as follows. As before, let $\eta(F)\cong F\times [-1,1]$ be a closed regular neighborhood of $F$ in $M_L$, and let $\eta(F)=\overline{M_L-N(F)}$. Choose a collection of mutually disjoint arcs $\gamma_1,\ldots,\gamma_n$ properly embedded in $F$ which cut $F$ into a disk. Assume that $\phi(\gamma_i)$ has been isotoped to intersect $\gamma_i$ transversely in the fewest possible points for each $i$.\\ 
	
	Let $\gamma'_1,\ldots,\gamma'_n$ be nearby parallel copies of $\gamma_1,\ldots,\gamma_n$ in $F$ such that $\gamma'_i\cap\gamma_j=\emptyset$ for all $i$ and $j$. For each $i$, we assume that $\gamma'_i$ has been chosen to intersect $\phi(\gamma_j)$ in the fewest possible points for every $j$. In particular, if $\phi(\gamma_i)$ lies on the same side of $\gamma_i$ near each endpoint, we choose $\gamma'_i$ so that it lies on the other side of $\gamma_i$. \\
	
	For each $i=1,\ldots,n$, we define $\alpha_i$ to be the boundary of the product disk $D(\gamma'_i)\subset N(F)$ and $\beta_i$ to be the boundary of $D(\phi(\gamma_i))\subset E(F)$. Thus, with respect to the product structure on $N(F)$, we have 
	\begin{equation}
	\alpha_i=\left(\gamma'_i\times\{-1,1\}\right)\cup\left(\partial\gamma'_i\times [-1,1]\right),\mbox{ }\mbox{ }\mbox{ }
	\beta_i=\left(\gamma_i\times\{1\}\right)\cup\left(\partial\gamma_i\times[-1,1]\right)\cup\left(\phi(\gamma_i)\times\{-1\}\right).
	\nonumber
	\end{equation}
	An example construction of one pair $\alpha_i,\beta_i$ is shown in Figure \ref{heegfig2}. By construction, both $\alpha=\{\alpha_1,\ldots,\alpha_n\}$, and $\beta=\{\beta_1,\ldots,\beta_n\}$ are cut systems for $\Sigma=\partial N(F)$, so  $(\Sigma,\alpha,\beta)$ is a Heegaard diagram for $M$. We say that such a Heegaard diagram is \textit{associated with $\phi$}.
\end{heegfib}

Two fundamental facts about Heegaard diagrams will be required. The ensuing discussion follows parts of Section 2 of \cite{MZ14}. To state the first, let $(\alpha,\beta)$ be a Heegaard diagram for a 3-manifold $M$, and let $\alpha=\{\alpha_1,\alpha_2\}$, $\beta=\{\beta_1,\beta_2\}$. Once the four curves have been oriented in some fashion, we consider the \textit{intersection matrix}
\begin{equation}
M(\alpha,\beta)=\begin{pmatrix} \alpha_1\cdot\beta_1 & \alpha_1\cdot\beta_2 \\ \alpha_2\cdot\beta_1 & \alpha_2\cdot\beta_2 \end{pmatrix}.
\nonumber
\end{equation}
This as a presentation matrix for the abelian group $H_1(M)$. One can see this through the description of cellular homology in terms of the handle decomposition of $M$ corresponding to the Heegaard diagram (see Section 4.2 of \cite{GS99}).\\

Consequently, if $H_1(M)$ is finite, then $|\det M(\alpha,\beta)|$ is equal to the order of $H_1(M)$. In particular, we have:

\begin{lem}\label{detlem}
	\cite{MZ14} If $(\alpha,\beta)$ is a Heegaard diagram for $S^3$, then $|\det M(\alpha,\beta)|=1$.
\end{lem}

\begin{figure}[t]
	\centering
	\includegraphics[scale=0.4]{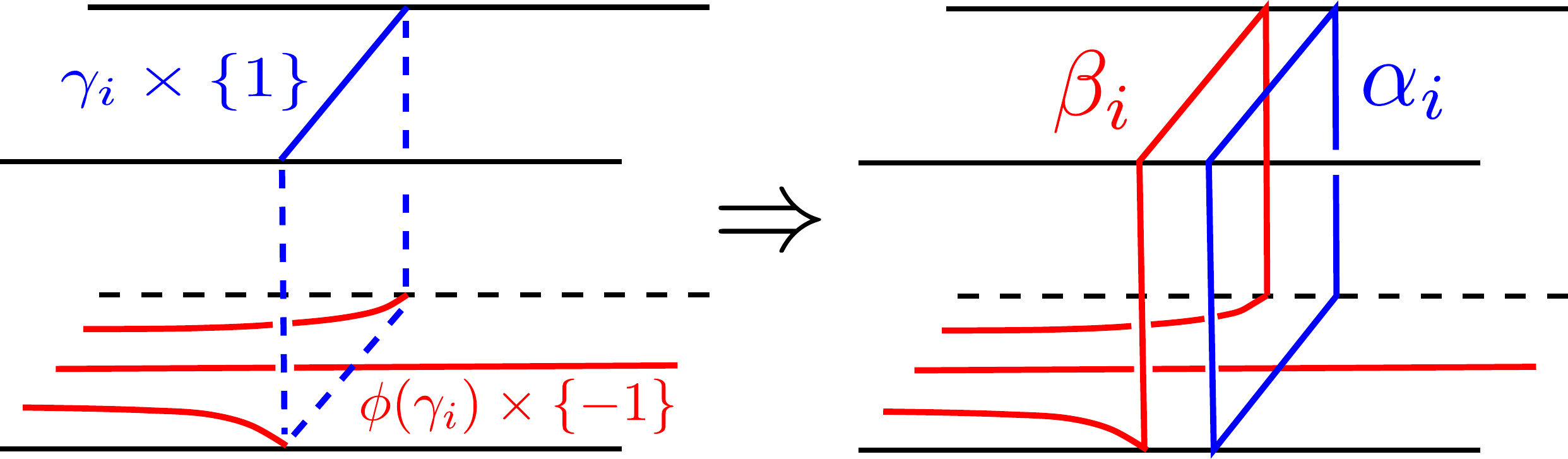}
	\caption{Constructing a pair of curves in a Heegaard diagram defined by the monodromy of a fiber surface.}
	\label{heegfig2}
\end{figure}

	We need some additional definitions. From this point forward, when considering two collections of curves in $\Sigma$, we assume that each curve in one collection has been isotoped to intersect each curve in the other in the fewest possible points. Thus, given a cut system $\alpha=\{\alpha_1,\alpha_2\}$ for $\Sigma$ and a collection $\mathcal{C}$ of mutually disjoint curves which are distinct from $\alpha_1$ and $\alpha_2$, we assume that the intersection of $\mathcal{C}$ with the cut surface $\Sigma_{\alpha}$ is a collection of properly embedded, essential arcs.

\begin{whitewave} 	
	Let $\alpha$ and $\mathcal{C}$ be as above. Viewing the boundary components of $\Sigma_{\alpha}$ as four fat vertices and the components of $\mathcal{C}\cap\Sigma_{\alpha}$ as edges, we may regard $\mathcal{C}\cap\Sigma_{\alpha}$ as a topological graph $\Sigma_{\alpha}(\mathcal{C})$, called the \textit{Whitehead graph} of $\mathcal{C}$ with respect to $\alpha$. Examples can be found in Figures \ref{whitehead_construct}, \ref{whitehead1}, and \ref{whitehead2} of the following section.\\
	
	Define a \textit{wave} in $\Sigma_{\alpha}$ to be an essential, properly embedded arc whose endpoints both lie in the same component of $\partial\Sigma_{\alpha}$. Given a Heegaard diagram $(\alpha,\beta)$, an \textit{$\alpha$-wave} for $(\alpha,\beta)$ is a wave $\omega$ properly embedded in $\Sigma_{\alpha}$ such that $\omega\cap\beta=\emptyset$. One similarly defines a \textit{$\beta$-wave} in $\Sigma_{\beta}$. If such an arc exists, we say that $(\alpha,\beta)$ contains a wave \textit{based in $\alpha$} (resp. $\beta$) and that $\Sigma_{\alpha}(\beta)$ (resp. $\Sigma_{\beta}(\alpha)$) \textit{contains a wave}.
\end{whitewave}

Our reason for considering waves in Heegaard diagrams is the following theorem of Homma, Ochiai, and Takahashi.

\begin{thm}\label{HOTthm}
	\cite{HOT80} Every genus two Heegaard diagram $(\alpha,\beta)$ for $S^3$ contains either an $\alpha$-wave or a $\beta$-wave. 
\end{thm}

The second half of our main proof to follow amounts to a direct application of this result, together with Lemma \ref{detlem}. It is a testament to the special nature of genus two Heegaard diagrams, as higher-genus Heegaard diagrams for $S^3$ need not contain any waves \cite{Och79}.

\section{Proof of Theorem \ref{mainthm}}

Our proof breaks into two main parts. In Section 3.1, we prove that the links listed in the statement of Theorem \ref{mainthm} are indeed fibered links of type $(0,3)$. We further show that the fiber surfaces of any two of these links are related by a sequence of Stallings twists, and describe their monodromies explicitly. In Section 3.2, we use the elements of Heegaard theory discussed in Section 2.2 to show that a fibered link of type $(0,3)$ in $S^3$ must have the same monodromy as one of these links. Since the monodromy determines the fiber surface up to isotopy \cite{BG16}, this completes the proof.

\subsection{Describing the fibrations}

The links listed in the statement of Theorem \ref{mainthm} are known to be fibered links of type $(0,3)$. The links $L_n$ of Figure \ref{gabai_surf} are discussed in \cite{Ga86}, and the example of Figure \ref{gen_zero_link} is exhibited in \cite{BIRS16} as the result of what the authors refer to as a generalized Hopf banding. However, we give a complete analysis, culminating in the needed descriptions of the monodromies of these links. \\

Our starting points are the basic facts that the right and left-handed Hopf links are both fibered links, with annular fiber surfaces, and that a boundary-connected sum of two fiber surfaces is again a fiber surface \cite{Ga86}. It follows immediately that all connected sums of two Hopf links are fibered links of type $(0,3)$, since their fiber surfaces are boundary-connected sums of two annuli.\\

The annular fiber surface of a Hopf link will be referred to as a \textit{Hopf annulus}. The Seifert surface of the link $L_n$ depicted in Figure \ref{gabai_surf} is obtained from a boundary-connected sum of two Hopf annuli of opposite types by performing a Stallings $-n$-twist along the curve shown on the left-hand side of Figure \ref{gabai_construct}, which is readily seen to be a Stallings curve of type 0. The reader can check that the result of pushing this curve off of the surface to either side is isotopic to that appearing on the right-hand side of the figure, surrounding the central band. Performing the $-n$-twist along the left-hand curve is therefore equivalent to performing $-1/n$ Dehn surgery along that on the right. The result of doing so is precisely the surface depicted in Figure \ref{gabai_surf}, so we conclude that it is a fiber surface.\\

\begin{figure}[t]
	\centering
	\includegraphics[scale=0.45]{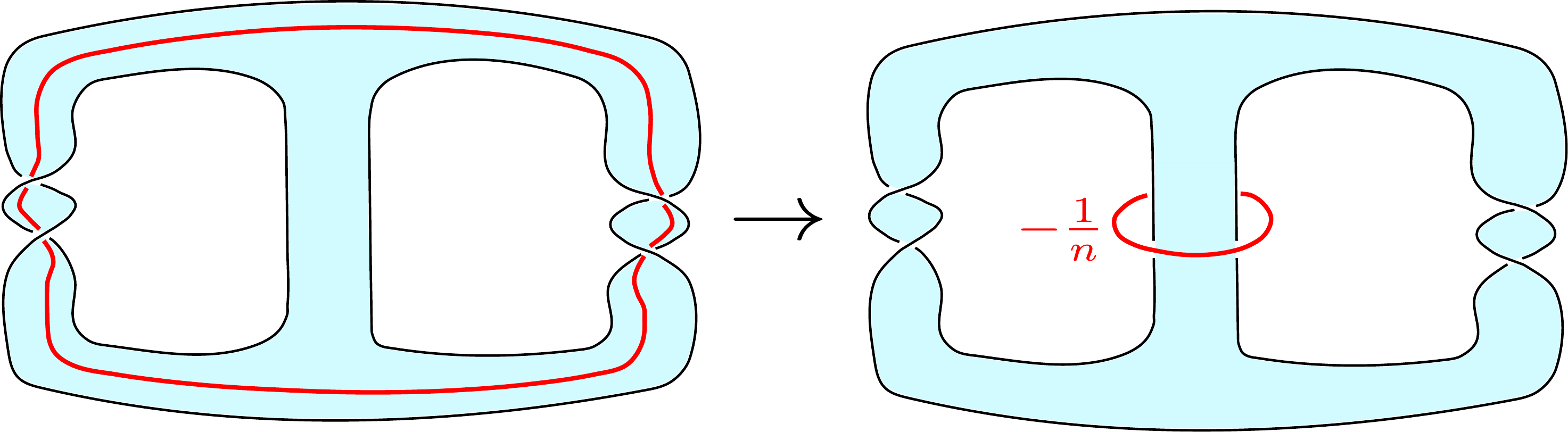}
	\caption{Constructing the fiber surface for $L_n$ of Figure \ref{gabai_surf} by performing a Stallings $n$-twist.}
	\label{gabai_construct}
\end{figure}

\begin{figure}[b]
	\centering
	\includegraphics[scale=0.45]{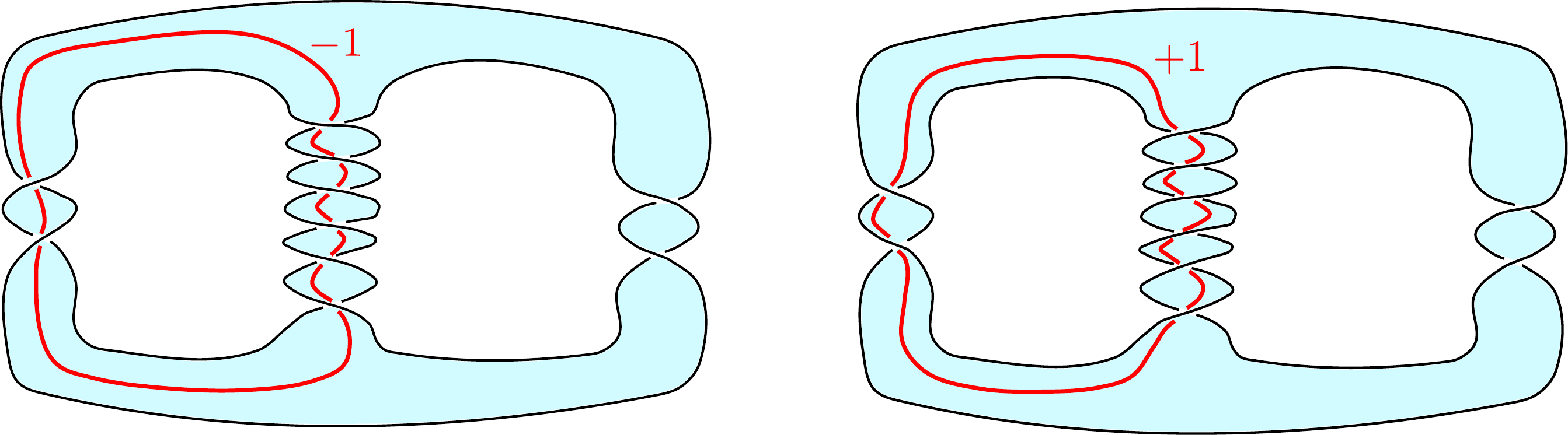}
	\caption{Constructing the fiber surfaces for the link of Figure \ref{gen_zero_link} (left) and its mirror image (right) via Stallings twists of types -1 and 1.}
	\label{special_construct}
\end{figure}

It remains to prove that the link $L$ of Figure \ref{gen_zero_link} and its mirror image $\overline{L}$ bound fiber surfaces of type $(0,3)$. Both can be constructed by performing Stallings twists on the fiber surfaces of $L_3$ and $L_{-3}$. To obtain the fiber surface for $L$, perform a positive Stallings twist on the fiber surface for $L_3$ along the red curve shown on the left of Figure \ref{special_construct}, which is a Stallings curve of type -1. The effect of the twist on the link can be seen by carefully blowing down this curve in the sense of Kirby calculus \cite{GS99}, viewed with framing equal to the surgery coefficient of $+1$. Doing so and simplifying yields the diagram of Figure \ref{gen_zero_link}. \\

The fiber surface for $\overline{L}$ is similarly obtained by performing a negative Stallings twist on the fiber surface for $L_{-3}$ along the curve shown on the right of Figure \ref{special_construct}. Since the total 4-component diagram is the mirror image of the previous one and the slope for the surgery curve is $-1$, it follows that the resulting link is indeed $\overline{L}$. This can also be seen from the fact, to be noted shortly, that the monodromies of this fiber surface and that of $L$ are inverse to each other.\\

\begin{figure}[t]
	\centering
	\includegraphics[scale=0.4]{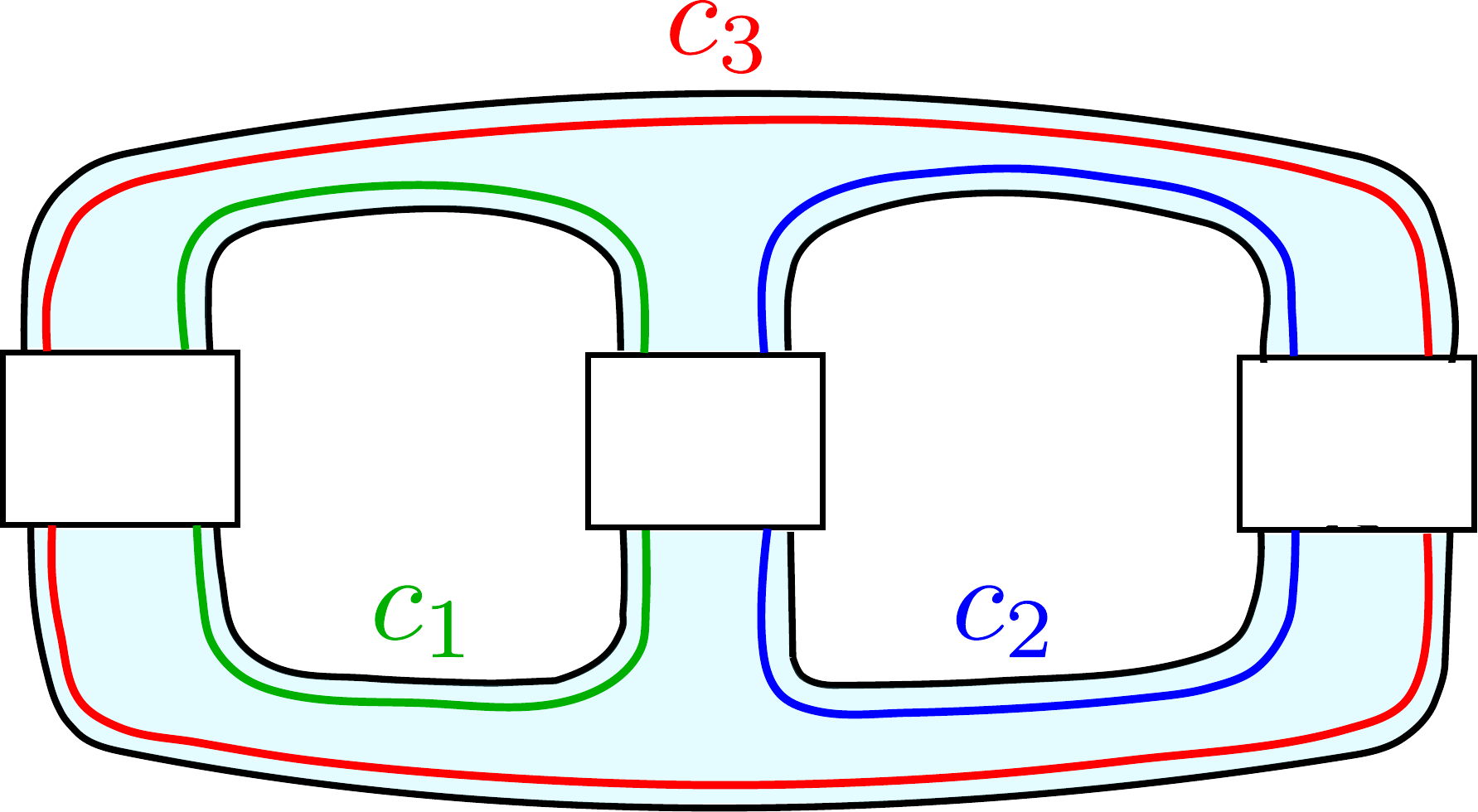}
	\caption{Labeling of the three essential curves in a surface of type $(0,3)$ used to write down monodromies of the fiber surfaces in the main theorem.}
	\label{mono_con}
\end{figure}

To set the stage for Part 2 of our proof, we determine the monodromies of all links examined above. Recall the notation $(M[n_1,n_2,n_3],L[n_1,n_2,n_3])$ for the manifold-link pair determined by the triple of integers $n_1,n_2,n_3$, as introduced in Section 2.1. This means that $L[n_1,n_2,n_3]$ is the fibered link of type $(0,3)$ in $M[n_1,n_2,n_3]$ whose fiber surface has monodromy of the form $T^{n_1}_1T^{n_2}_2T^{n_3}_3$, where $c_1$, $c_2$, and $c_3$ are disjoint curves in the fiber parallel to its distinct boundary components and $T_i$ is a left-handed Dehn twist along $c_i$. For simplicity, we will denote this monodromy by $\phi[n_1,n_2,n_3]$. For the present discussion, we choose the labeling of $c_1$, $c_2$, and $c_3$ relative to the diagrams of Figures \ref{gabai_construct} and \ref{special_construct} as shown in Figure \ref{mono_con}. However, as discussed in Section 2.1, $[n_1,n_2,n_3]$ should really be regarded as a multiset of three integers, rather than an ordered triple.\\

In the case of a boundary-connected sum of two Hopf annuli, our convention is such that $c_3$ is the only curve which intersects both summands. Since the monodromy of a right-handed (resp. left-handed) Hopf annulus is a left (resp. right) Dehn twist about its core curve, it follows that the monodromy for a boundary-connected sum of two Hopf annuli is $\phi[\pm 1,\pm 1,0]$, as its restriction to each summand must be the monodromy of that summand \cite{Ga86}. When the two Hopf annuli are of opposite types, as on the left of Figure \ref{gabai_construct}, the monodromy is $\phi[1,-1,0]$. It then follows from the above construction of the links $L_n$, via the final remark of Section 2.1, that the monodromy for the fiber surface of $L_n$ is given by $\phi[1,-1,-n]$. Likewise, by our construction of the remaining two fiber surfaces of type $(0,3)$ from those of $L_{-3}$ and $L_3$ via Stallings twists, it follows that their monodromies are given by $\phi[2,-1,3]$ and $\phi[-2,1,-3]$ (respectively).\\

Using the above observations, we can now show that any two of the above links are related by a sequence of Stallings twists. So far, we have seen that the fiber surfaces of the links $L_n$, $L$, and $\overline{L}$ can be constructed from that of $L_0$ by a sequence of Stallings twists. The claim therefore reduces to showing that a boundary-connected sum of two Hopf annuli of the same type is related to the fiber surface of one of these links by a sequence of Stallings twists. In fact, one twist will suffice. Letting $F$ denote one of these two surfaces, a curve $c$ in $F$ parallel to the boundary component which meets both of the Hopf annulus summands will be a Stallings curve of type $1$ if they are right-handed, and type $-1$ if they are left-handed. Performing the corresponding Stallings twist produces a fiber surface with monodromy $\phi[1,1,-1]=\phi[1,1,0]\circ T^{-1}_c$ in the first case, and one with monodromy $\phi[-1,-1,1]=\phi[-1,-1,0]\circ T_c$ in the second. It follows from the above discussion that the resulting fiber surfaces are those of $L_{-1}$ and $L_1$ (respectively). We have therefore proven the second statement of Theorem \ref{mainthm}.

\subsection{The list is complete}

It remains to show that we have described all fibered links of type $(0,3)$. By the work done in Section 3.1, it suffices to prove that $M[n_1,n_2,n_3]\cong S^3$ only if $[n_1,n_2,n_3]$ is one of $[\pm 1,\pm 1,0]$, $[2,-1,3]$, $[-2,1,-3]$, or $[1,-1,n]$. To begin, note that $M[-n_1,-n_2,-n_3]$ is related to $M[n_1,n_2,n_3]$ by an orientation-reversing homeomorphism which takes $L[-n_1,-n_2,-n_3]$ to $L[n_1,n_2,n_3]$, as the corresponding monodromies are inverse to each other. We may consequently assume that at least two of the integers, say $n_1$ and $n_3$, are nonnegative. \\

If one of the $n_i$ is zero, then the monodromy for $F[n_1,n_2,n_3]$ fixes an essential, separating arc, revealing a 2-sphere which decomposes $F[n_1,n_2,n_3]$ as the boundary-connected sum of two annular fiber surfaces. If $M[n_1,n_2,n_3]\cong S^3$, this means that $F[n_1,n_2,n_3]$ is a boundary-connected sum of two Hopf annuli, in which case $[n_1,n_2,n_3]=[\pm 1,\pm 1,0]$. Thus, we need only consider the case that $n_1$ and $n_3$ are strictly positive. If in addition $n_2=-1$, we already know that $M[n_1,n_2,n_3]\cong S^3$ if either $n_1=1$ or $n_3=1$, as $M[n_1,n_2,n_3]$ is then $M[1,-1,n]$ for some $n$. It therefore suffices to prove the following.

\begin{propn}\label{mainprop}
(a) If $n_1>0$, $n_3>0$, and $M[n_1,n_2,n_3]\cong S^3$, then $n_2=-1$.\\
	
(b) If $n_2=-1$ and $1<n_1\leq n_3$, then $M[n_1,n_2,n_3]\cong S^3$ if and only if $n_1=2$ and $n_3=3$.
\end{propn}

Part (b) of this proposition will be established in Lemma \ref{techlem} by applying Lemma \ref{detlem}. There, we also take the first step towards proving (a) by showing that $n_2<0$ if both $n_1$ and $n_3$ are positive. We then complete the proof (a) by applying Theorem \ref{HOTthm}. \\

\begin{figure}[t]
	\centering
	\includegraphics[scale=0.45]{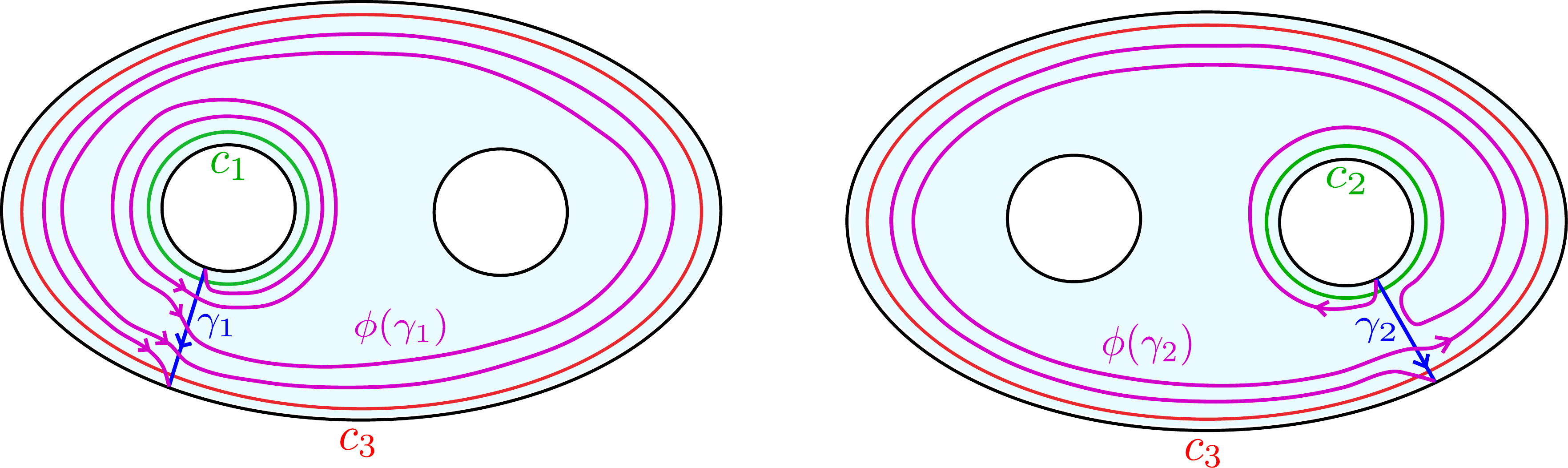}
	\caption{Images of the arcs $\gamma_1$ and $\gamma_2$ under $\phi=T^2_1T^{-1}_2T^2_3$.}
	\label{mono_gamma1}
\end{figure}

For what follows, $F$ will denote the (abstract) surface of type $(0,3)$, with push-offs of its boundary components labeled $c_1,c_2,c_3$ in some fashion. In the above notation, this labeling yields a description of the monodromy of $L[n_1,n_2,n_3]$ as $\phi[n_1,n_2,n_3]=T^{n_1}_1T^{n_2}_2T^{n_3}_3$. Let $\gamma_1$ and $\gamma_2$ be mutually disjoint arcs properly embedded in $F$ such that $\gamma_i$ runs from the boundary component parallel to $c_i$ to that parallel to $c_3$. We consider the genus two Heegaard diagram $(\alpha,\beta)$ determined by $\phi[n_1,n_2,n_3]$, as defined in Section 2.2. Recall that we are assuming both $n_1$ and $n_3$ to be positive.\\

We first describe the action of $\phi=\phi[n_1,n_2,n_3]$ on $\gamma_1$ and $\gamma_2$. As described in \cite{FM12} in a more general setting, the process of computing $\phi(\gamma_i)$ may be viewed as taking $|n_j|$ parallel copies of $c_j$ for each $j=1,2,3$ and resolving their points of intersection with $\gamma_i$ in the appropriate manner, as determined by the sign of each $n_j$. In both cases, after $\phi(\gamma_i)$ is computed in this fashion and perturbed very slightly, it will intersect both $\gamma_1$ and $\gamma_2$ in the fewest possible number of points. The computations for the case $n_1=2=n_3$, $n_2=-1$ are shown in Figure \ref{mono_gamma1}. Note that, under the conventions indicated in that figure, $\phi(\gamma_i)$ will intersect $\gamma_i$ in $|n_i|+n_3$ points. \\

Let $(\Sigma,\alpha,\beta)$ denote the Heegaard diagram associated with $\phi$ determined by the pair $\gamma_1,\gamma_2$, as defined in Section 2.2. We begin by using the intersection matrix $M(\alpha,\beta)$ to obtain some basic restrictions on the triple $n_1,n_2,n_3$ needed for $M[n_1,n_2,n_3]$ to be $S^3$. In particular, we prove part (b) of Proposition \ref{mainprop}.

\begin{figure}[t]
	\centering
	\includegraphics[scale=0.5]{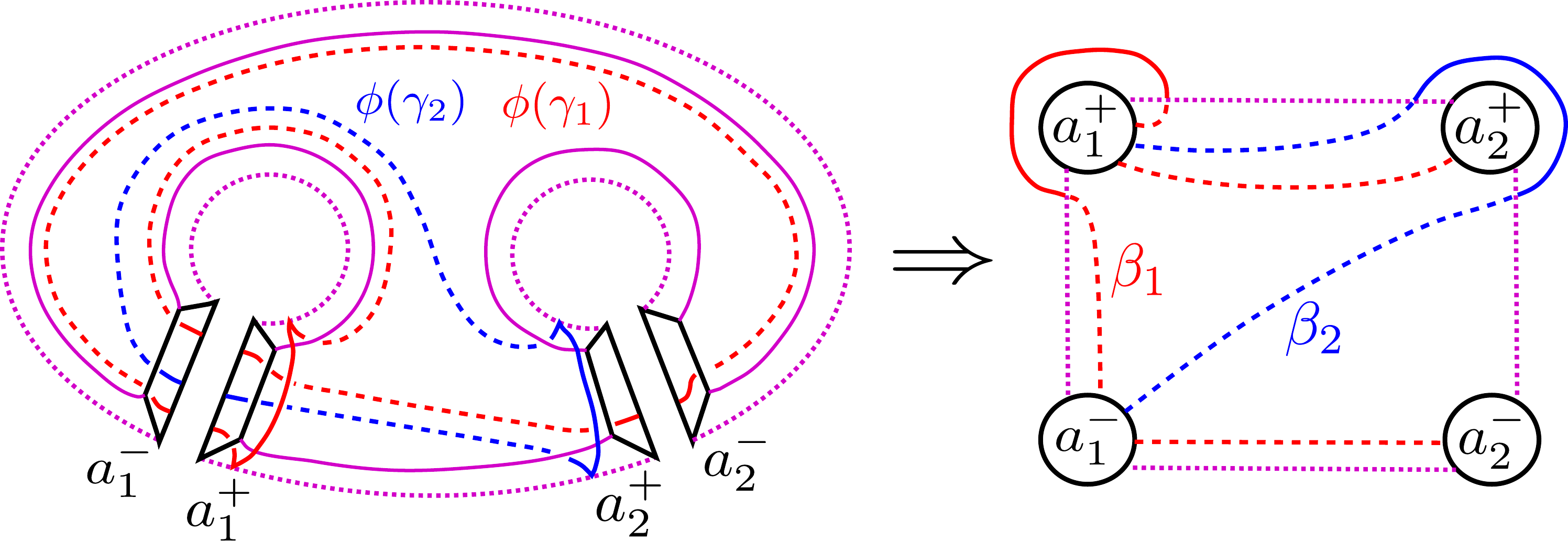}
	\caption{Using $\phi(\gamma_1)$ and $\phi(\gamma_2)$ to draw the Whitehead graph for $(\alpha,\beta)$. On the left, the dashed arcs should be thought of as sitting in $F\times\{-1\}\subset \Sigma=\partial\left(F\times[-1,1]\right)$, with the solid arcs lying above them. The dotted purple arcs lie in the boundary of $F\times\{-1\}$, and are not part of the Whitehead graph.}
	\label{whitehead_construct}
\end{figure}

\begin{lem}\label{techlem}
	Let $M[n_1,n_2,n_3]$ be the 3-manifold defined above. Suppose that $n_1$ and $n_3$ are strictly positive and $M[n_1,n_2,n_3]\cong S^3$. If $n_2\neq 0$, then $n_2$ must be negative. Further, if $n_2=-1$ and $1<n_1\leq n_3$, then $n_1=2$ and $n_3=3$.
\end{lem}

\begin{proof}[\textbf{Proof}]
	Let $(\Sigma,\alpha,\beta)$ be as above. To write down $M(\alpha,\beta)$, for both $i=1,2$, orient both $\gamma_i$ and $\phi(\gamma_i)$ so that their terminal endpoints lie on the boundary curve parallel to $c_3$. Orient the corresponding curves $\alpha_i$ and $\beta_i$ in the Heegaard diagram accordingly.\\
	
	As visible in Figure \ref{heegfig2} of Section 2.2, for both $i=1,2$, the points in which $\alpha_i$ and $\beta_i$ intersect all lie in the lower copy $F\times\{-1\}\subset\Sigma$ of the fiber surface. In fact, they correspond exactly to the $|n_i|+n_3$ points of intersection of $\gamma_i$ with $\phi(\gamma_i)$ if $n_i$ is positive, endpoints included, while they correspond to the interior points of intersection of the two arcs if $n_i$ is negative. In the former case, as illustrated on the left side of Figure \ref{mono_gamma1}, the sign of each intersection point of $\gamma_i$ with $\phi(\gamma_i)$ is + when the unit tangent vector to $\gamma_i$ is taken as the first vector in the corresponding basis of the tangent space. Since the signs of these points are the same as the signs of the corresponding intersection points of $\alpha_i$ with $\beta_i$, this means that $\alpha_i\cdot\beta_i=|\gamma_i\cap\phi(\gamma_i)|=n_i+n_3$ when $n_i$ is positive.\\
	
	It also turns out that $\alpha_i\cdot\beta_i=n_i+n_3$ when $n_i$ is negative. Here, while $\alpha_i$ intersects $\beta_i$ in two fewer points, the pair of points of $\gamma_i\cap\phi(\gamma_i)$ that have been removed are the endpoints. Since $n_i$ is negative while $n_3$ is positive, these have opposite signs. The signs of the points of $Int(\gamma_i)\cap Int(\phi(\gamma_i))$, measured as before, match the signs of the corresponding points of $\alpha_i\cap\beta_i$. This means that $\alpha_i\cdot\beta_i$ is still equal to the sum of the signs of the points of $\gamma_i\cap\phi(\gamma_i)$. The first $|n_i|$ of these points encountered as $\gamma_i$ is traversed, including the initial endpoint, all have negative sign, while the final $n_3$ intersection points all have positive sign. Thus, we have $\alpha_i\cdot\beta_i=-|n_i|+n_3=n_i+n_3$. \\
	
	It follows from our orientation conventions that the remaining two entries $\alpha_1\cdot\beta_2$ and $\alpha_2\cdot\beta_1$ of $M(\alpha,\beta)$ are both equal to $n_3$, since $n_3>0$. By the observations of the previous two paragraphs, we therefore have 
	\begin{equation}
	M(\alpha,\beta)=\begin{pmatrix} n_1+n_3 & n_3 \\ n_3 & n_2+n_3 \end{pmatrix}
	\nonumber
	\end{equation}
	regardless of the sign of $n_2$. By Lemma \ref{detlem}, the fact that $M[n_1,n_2,n_3]\cong S^3$ implies 
	\begin{equation}
	1=|\det M(\alpha,\beta)|=|n_1n_2+n_1n_3+n_2n_3|.
	\nonumber
	\end{equation}
	Since both $n_1$ and $n_3$ are positive, it is evident that the right-hand side of this equation is greater than 1 if $n_2$ is also positive, so $n_2$ must be negative, as claimed. When $n_2=-1$, this equation can be written as $1=|(n_1-1)(n_3-1)-1|$, which means that $(n_1-1)(n_3-1)$ is equal to either 0 or 2. If both $n_1$ and $n_3$ are greater than 1, in which case this expression is non-zero, it follows immediately that $n_1=2$ and $n_3=3$, since we have assumed $n_1\leq n_3$.
\end{proof}

\begin{proof}[\textbf{Proof of (a) of Proposition \ref{mainprop}}]
Suppose that $M[n_1,n_2,n_3]\cong S^3$, $n_1$ and $n_3$ are strictly positive, and $n_2\neq 0$. By Lemma \ref{techlem}, we know that $n_2<0$. We wish to prove that $n_2=-1$. To do this, we consider the Whitehead graph for the Heegaard diagram of $M[n_1,n_2,n_3]$ discussed above. Our method of drawing $\Sigma_{\alpha}(\beta)$ is illustrated in Figure \ref{whitehead_construct} for the simple case $n_1=1$, $n_2=-1$, $n_3=1$. As in Section 2.2, we view the Heegaard surface $\Sigma$ as the boundary of $F\times[-1,1]$, where $F$ is a surface of type $(0,3)$ identified with the fiber surface for $L[n_1,n_2,n_3]$. Letting $\phi$ denote the corresponding monodromy, we compute $\phi(\gamma_1)$ and $\phi(\gamma_2)$ as before, where all arcs are thought of as lying in the `bottom' $F\times\{-1\}$ of $\Sigma$. We then construct $\beta_i$ and $\alpha_i$ from $\phi(\gamma_i)$ and a push-off of $\gamma_i$ (respectively) as described in Section 2.2, and cut $\Sigma$ along $\alpha_1$ and $\alpha_2$ to obtain $\Sigma_{\alpha}$. In the figure, the boundary components of $\Sigma_{\alpha}$ arising from the cut along $\alpha_i$ are denoted by $a^-_i$ and $a^+_i$, labeled so that the arc in $\beta_i$ which lies outside of $F\times\{-1\}$ runs parallel to an arc in $a^+_i$.\\

\begin{figure}
	\centering
	\includegraphics[scale=0.6]{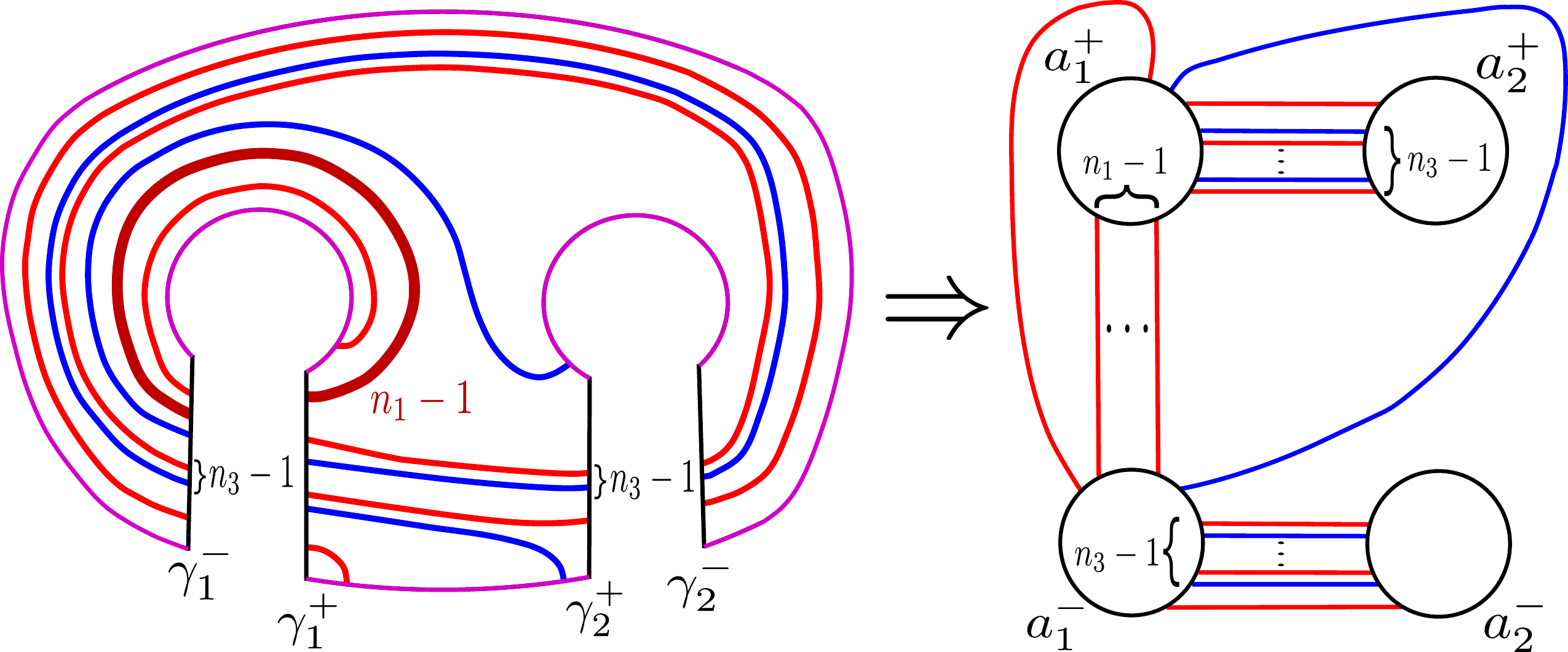}
	\caption{Constructing the Whitehead graph $\Sigma_{\alpha}(\beta)$ in the case $n_1,n_3>0$, $n_2=-1$. On the right, the weights given to the bold arc and two pairs of parallel arcs indicate the number of times that they appear.}
	\label{whitehead1}
\end{figure}

This method of constructing the Whitehead graph $\Sigma_{\alpha}(\beta)$ is illustrated in Figures \ref{whitehead1} and \ref{whitehead2} for the pertinent values of $n_1$, $n_2$, and $n_3$. In each figure, the diagram on the left corresponds to the portion of the left side of Figure \ref{whitehead_construct} which lies in the lower copy $F\times\{-1\}\subset\Sigma$ of $F$. While one can construct a picture of the other Whitehead graph $\Sigma_{\beta}(\alpha)$ directly, there is a way of viewing this process which 

\begin{figure}[h]
	\centering
	\includegraphics[scale=0.6]{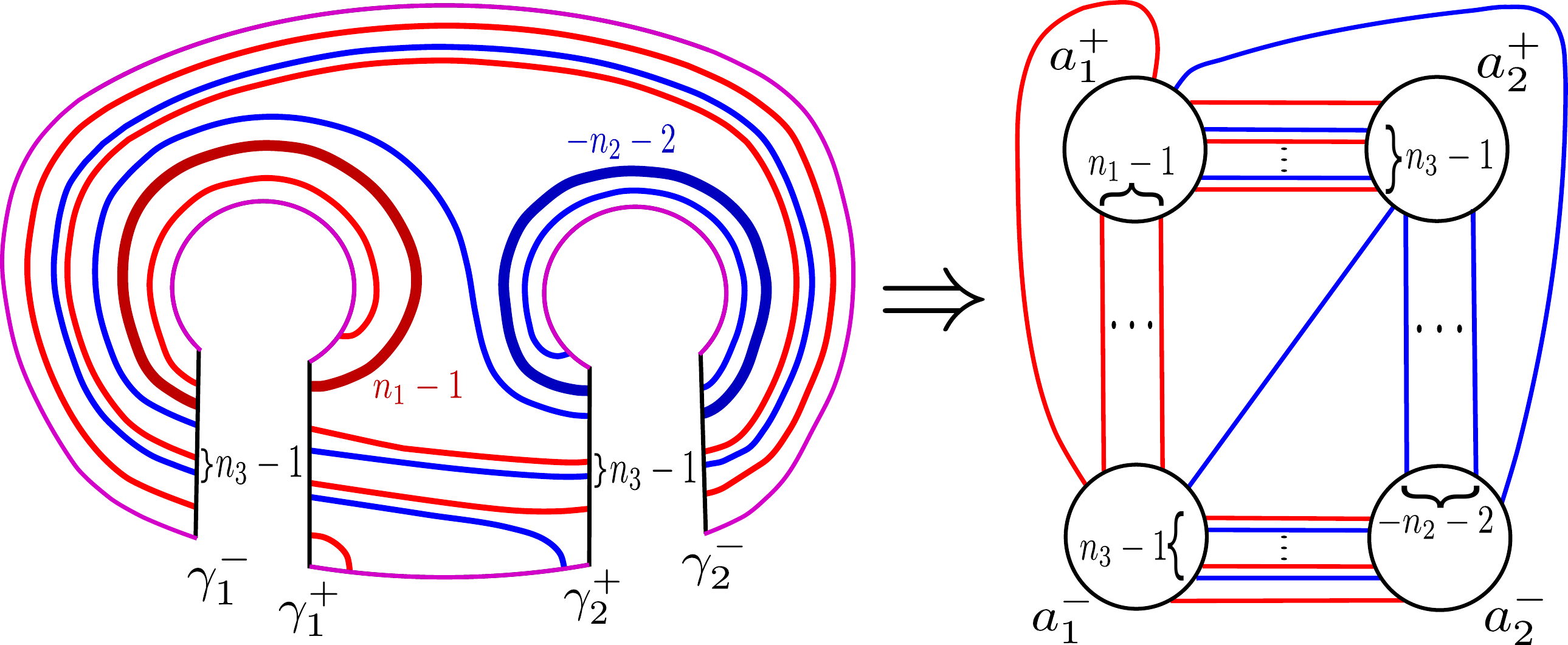}
	\caption{Constructing the Whitehead graph $\Sigma_{\alpha}(\beta)$ in the case $n_1,n_3>0$, $n_2<-1$, where it does not contain a wave. On the right, the weights given to the two bold arcs and pairs of parallel arcs indicate the number of times that they appear.}
	\label{whitehead2}
\end{figure}

\noindent
reveals a direct relationship between $\Sigma_{\beta}(\alpha)$ and $\Sigma_{\alpha}(\beta)$. Begin by applying $\phi^{-1}=T^{-n_1}_1T^{-n_2}_2T^{-n_3}_3$ simultaneously to all four arcs involved in the construction of $\Sigma_{\alpha}(\beta)$. We can reconstruct $\alpha_i$ and $\beta_i$ from this new picture as before, as illustrated in Figure \ref{heegfig2} of Section 2.2, but with $\phi^{-1}(\gamma_i)$ playing the role of $\phi(\gamma_i)$ and the roles of $\alpha_i$ and $\beta_i$ interchanged. We then cut along $\beta_1$ and $\beta_2$ to obtain $\Sigma_{\beta}(\alpha)$.\\

Under appropriate labellings of the copies of $\beta_i$ in $\Sigma_{\beta}$ as $b^+_i$ and $b^-_i$ for $i=1,2$, this construction reveals $\Sigma_{\beta}(\alpha)$ to be exactly the same as $\Sigma_{\alpha}(\beta)$ upon replacing $b$'s with $a$'s. To see this, note that the above procedure is tantamount to taking our original picture of $F$, with the arcs $\gamma_i$ and $\phi(\gamma_i)$, flipping it over vertically, and then proceeding just as in the construction of $\Sigma_{\alpha}(\beta)$. The correspondence between the pictures in $F-(N(\gamma_1)\cup N(\gamma_2))$ leading to the Whitehead graphs is depicted in Figure \ref{alpha_beta} for the case $n_1=1=n_3$, $n_2=-1$. The upper diagram is that used to construct $\Sigma_{\beta}(\alpha)$, while the lower two diagrams are used to construct $\Sigma_{\alpha}(\beta)$.\\

\begin{figure}[!]
	\centering
	\includegraphics[scale=0.35]{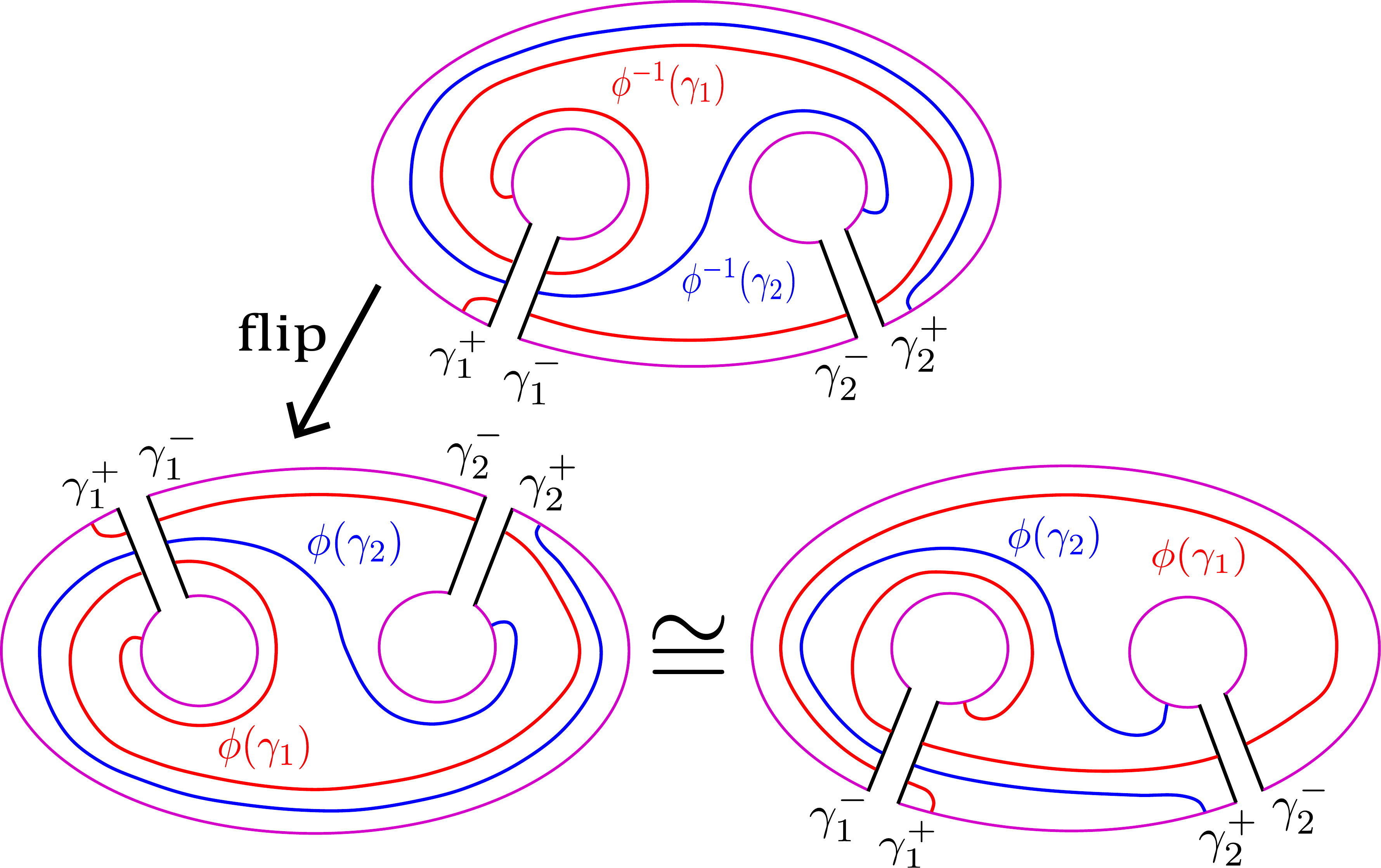}
	\caption{Going from the diagram in $F$ used to construct $\Sigma_{\beta}(\alpha)$ (top) and that used to construct $\Sigma_{\alpha}(\beta)$ (bottom) in the case $n_1=1=n_3$, $n_2=-1$, the same one treated in Figure \ref{whitehead_construct}.}
	\label{alpha_beta}
\end{figure}

In particular, we see that $\Sigma_{\alpha}(\beta)$ contains a wave if and only if $\Sigma_{\beta}(\alpha)$ does. By Theorem \ref{HOTthm}, $\Sigma_{\alpha}(\beta)$ must contain a wave if $M[n_1,n_2,n_3]\cong S^3$. However, as visible in Figure \ref{whitehead2}, $\Sigma_{\alpha}(\beta)$ does not contain a wave when $n_1>0$, $n_3>0$, and $n_2<-1$, so it follows that $\Sigma_{\beta}(\alpha)$ also does not contain a wave. Combined with Lemma \ref{techlem}, this shows that if $n_1>1$, $n_3>1$, and $n_1\leq n_3$, then $M[n_1,n_2,n_3]\cong S^3$ only if $n_2=-1$. This proves Proposition \ref{mainprop} (a), and thereby proves the main theorem.
\end{proof}

\bibliography{DissertationBibliography}

\bibliographystyle{siam}

\end{document}